\documentclass[12pt,a4paper]{article}
\usepackage[utf8]{inputenc}
\usepackage[T1]{fontenc}
\usepackage{amsmath}
\usepackage{amssymb}
\usepackage{amsthm}
\usepackage{graphicx}
\usepackage{mathrsfs}
\usepackage{float}
\usepackage{mathrsfs}
\usepackage{xcolor}
\usepackage{hyperref}
\usepackage{cite}
\usepackage{tikz}
\usepackage{array}
\usepackage[margin=1.3in]{geometry}
\newtheorem{corr}{Corollary}[section]
\newtheorem{example}{Example}[section]
\newtheorem{proposition}{Proposition}[section]

\newtheorem{theorem}{Theorem}[section]
\newtheorem{lemma}{Lemma}[section]
\newtheorem{definition}{Definition}[section]
\title{Lambda number of the reduced power graph of finite groups}
\author{\normalsize Siddharth Malviy$^a$, Keerttiranjan Sahoo$^a$, Vipul Kakkar$^a$ \vspace{0.2cm}\\ \em\small\ $^a$Department of Mathematics, Central University of Rajasthan, Ajmer, India}
\date{}
\begin{document}
	\maketitle
	\noindent \textbf{{Abstract.}} The reduced power graph $\Gamma_{RP}(G)$ of a group $G$ is a simple graph with vertex set $G$, where two distinct vertices $u$ and $v$ are adjacent if and only if $\langle u \rangle \subset \langle v \rangle$ or $\langle v \rangle \subset \langle u \rangle$. An $L(2,1)$-labeling of a graph $\Gamma$ is a mapping from the vertex set of $\Gamma$ to the set
of non-negative integers such that the labels of adjacent vertices differ by at least
$2$ and the labels of vertices at distance two are distinct. The {lambda number}
of $\Gamma$, denoted by $\lambda(\Gamma)$, is defined as the minimum possible span of labels among all $L(2,1)$-labelings of $\Gamma$.
In this paper, we establish upper and lower bounds for $\lambda(\Gamma_{RP}(G))$ and characterize
the conditions under which these bounds are achieved. As an application, we determine the exact value of $\lambda(\Gamma_{RP}(G))$ for certain classes of finite groups.\\
    
	

    \section{Introduction} 
     The fusion of group theory and graph theory has emerged as an important area of research within combinatorial algebra. By associating graphs to algebraic structures such as groups, one can obtain valuable insights into the structural properties of the group and the relationships among its elements.

\noindent A notable example of such an association is the power graph of a group, introduced by Kelarev in \cite{1}, where he studied the directed power graph of a semigroup. Subsequently, Chakrabarty~\cite{2} investigated the undirected power graph of semigroups and groups. Since then, power graphs of groups have attracted considerable attention, and several researchers have contributed to their structural and combinatorial analysis~\cite{3,4}.

\noindent We first recall that a graph $\Gamma$ is an ordered pair $\Gamma = (V,E)$, where $V(\Gamma)$ and $E(\Gamma)$ denote its set of vertices and set of edges, respectively. The number of vertices $|V (\Gamma)|$ is called the order of the graph $\Gamma$.  Two distinct vertices $u_1, u_2 \in V(\Gamma)$ are said to be {adjacent}, written $u_1 \sim u_2$, if $\{u_1,u_2\} \in E(\Gamma)$; otherwise, they are {non-adjacent}, denoted by $u_1 \nsim u_2$. A graph is called simple if it has no loops and
parallel edges.

\noindent To reduce the complexity arising from the large number of arcs and edges in power graphs (cf. \cite{ma2021lambda} for the
definition and other properties of a power graph of a group), R. Rajkumar and T. Anitha~\cite{rajkumar2017reduced} introduced the concept of the reduced power graph. The reduced power graph $\Gamma_{RP}(G)$ of a group $G$ is a simple graph with vertex set $G$, where two distinct vertices $u$ and $v$ are adjacent if and only if $\langle u \rangle \subset \langle v \rangle$ or $\langle v \rangle \subset \langle u \rangle$, where $\langle u \rangle \subset \langle v \rangle$ denotes that the cyclic group generated by $u$ is a proper subgroup of the cyclic group generated by $v$. Several structural properties of reduced power graphs have been studied in~\cite{rajkumar2018some}.

\noindent For vertices $u, v \in V(\Gamma)$, the {distance} $d(u,v)$ is defined as the length (number of edges) of the shortest path between $u$ and $v$.
\noindent For non–negative integers $j$ and $k$, an $L(j,k)$–labeling of a graph $\Gamma$ is a function 
$f : V(\Gamma) \to \mathbb{Z}$ that satisfies the conditions:  
\[
|f(u)-f(v)| \geq k \quad \text{if } d(u,v)=2, 
\qquad |f(u)-f(v)| \geq j \quad \text{if } u \sim v.
\]
The \emph{span} of $f$ is defined as $(\max f - \min f)$. By convention, we take $\min f = 0$, so the span becomes simply $\max f$.  
The $L(j,k)$–labeling number of $\Gamma$, denoted by $\lambda_{j,k}(\Gamma)$, is the minimum span among all $L(j,k)$–labelings of $\Gamma$.  

The case \(j=2\) and \(k=1\) has been studied most extensively, where \(\lambda_{2,1}(\Gamma)\), commonly denoted by \(\lambda(\Gamma)\), is called the \emph{\(\lambda\)-number} of graph \(\Gamma\). The lambda number of the complete graph of order $n$ is $2n-2$.
This problem is motivated by applications in radio channel assignment \cite{hale1980frequency} and in optical networks \cite{roberts1991t}.  
The concept of $L(2,1)$–labeling was first proposed by Griggs and Yeh \cite{griggs1992labelling}, who showed that determining such a labeling for a general graph is NP–complete.  
Georges and Mauro \cite{georges2003labeling} later extended this framework to the general $L(j,k)$–labeling problem. Several works have contributed to this area (see \cite{georges2003regular, georges1994relating, kim2017lambda, whittlesey1995lambda}), and a detailed survey of established results and open problems is provided in \cite{yeh2006survey}.  

Graphs constructed from groups and related algebraic structures have also been explored in this setting due to their importance and their links with automata theory (cf. \cite{kelarev2003graph, kelarev2009cayley, kelarev2002ring, kelarev2004labelled}).  
Zhou \cite{zhou2005labelling} investigated $L(j,k)$–labelings of Cayley graphs of the abelian groups, while Kelarev et al. \cite{kelarev2015distance} studied connections between semigroup properties and the minimum spans of distance labelings of their Cayley graphs.  
The $\lambda$–number of power graphs of finite groups has been analyzed in \cite{ma2021lambda, sarkar2022lambda, sarkar2023lambda} and that of enhanced power graphs of finite groups in \cite{parveen2025lambda}.

\noindent Throughout this paper, $G$ is a finite group with the identity element $e$.

\noindent Motivated by the work of the lambda number, in this paper we study the lambda number of the reduced power graphs of finite groups.\\
\noindent The paper is organized as follows. In Section~\ref{sec2}, we present the basic definitions, notations, and preliminary results used throughout the paper. In Section~\ref{sec3}, we first study the $\lambda$-number of an arbitrary graph $\Gamma$ and then apply these results to determine the $\lambda$-number of the reduced power graph $\Gamma_{RP}(G)$.
%\noindent First, we identify some groups (Example \ref{ex1}, \ref{p group qn lambda}, \ref{ex3}) for which the $\lambda$-numbers of their reduced power graphs achieve the lower bound stated in Theorem~\ref{lower bound main thm}. 

    \section{Preliminaries} \label{sec2}

    In this part, we briefly recall some standard graph-theoretic terms.

A {subgraph} $\Gamma'$ of $\Gamma$ is a graph such that $V(\Gamma') \subseteq V(\Gamma)$ and $E(\Gamma') \subseteq E(\Gamma)$. If in addition $V(\Gamma') = V(\Gamma)$, then $\Gamma'$ is called a {spanning subgraph} of $\Gamma$.  

The {complement} of $\Gamma$, written $\overline{\Gamma}$, is the graph on the same vertex set as $\Gamma$ where two distinct vertices are adjacent in $\overline{\Gamma}$ exactly when they are not adjacent in $\Gamma$.  

A simple graph is said to be {complete} if every pair of distinct vertices is adjacent; the complete graph with $n$ vertices is denoted by $K_n$.  

%A graph $\Gamma$ is called {$k$-partite} if its vertex set can be partitioned into $k$ disjoint subsets such that no two vertices within the same subset are adjacent. The \textit{complete $k$-partite graph} with part sizes $n_1,n_2,\dots,n_k$ is written as $K_{n_1,n_2,\dots,n_k}$, where each vertex in one part is adjacent to all vertices of the other parts.  

A vertex $v \in V(\Gamma)$ is termed a {universal vertex} if it is adjacent to every other vertex of the graph. The set of all universal vertices of $\Gamma$ is denoted by $\mathrm{Univ}(\Gamma)$.  

A graph $\Gamma$ is said to be {connected} if, for every pair of vertices, there exists a path linking them; otherwise, $\Gamma$ is {disconnected}.

Let $\Gamma_1=(V_1,E_1)$ and $\Gamma_2=(V_2,E_2)$ be two graphs.  
The {union} of $\Gamma_1$ and $\Gamma_2$, denoted by $\Gamma_1 \cup \Gamma_2$, is the graph with  \(V(\Gamma_1 \cup \Gamma_2) = V_1 \cup V_2 \quad \text{and} \quad E(\Gamma_1 \cup \Gamma_2) = E_1 \cup E_2\). The {join} of $\Gamma_1$ and $\Gamma_2$, denoted by $\Gamma_1 + \Gamma_2$, is obtained from $\Gamma_1 \cup \Gamma_2$ by adding all possible edges between vertices of $V_1$ and vertices of $V_2$, that is  
\(E(\Gamma_1 + \Gamma_2) = E_1 \cup E_2 \cup \{ uv : u \in V_1,\, v \in V_2 \}\).

A {path covering} $c(\Gamma)$ of a graph $\Gamma$ is a family of vertex-disjoint paths such that every vertex of $\Gamma$ belongs to one of these paths. The {path covering number} $c(\Gamma)$ is the smallest cardinality among all such coverings.

\noindent In order to determine the $\lambda$–number of our graphs, we recall a previous result that will be useful.

\begin{lemma} \label{thm 1} \cite[Theorem 1.1] {georges1994relating}
    Let $\Gamma$ be a graph with $n$ vertices.  Then 
\begin{enumerate}
    \item  $\lambda(\Gamma) \leq n-1$ if and only if $c(\overline{\Gamma})=1$.  
    \item  for any integer ${r} \geq 2$, $\lambda(\Gamma) = n+ {r}-2$ holds if and only if $c(\overline{\Gamma})={r}$.
\end{enumerate}
\end{lemma}

\noindent %For an element $x \in G$, its centralizer is denoted by $C_G(x)$, while the center of $G$ is written as $Z(G)$. 
The dihedral group $D_n$ of order $2n$ is presented by
\[
D_n=\langle x,y \mid x^n=e,\; y^2=e,\; yxy^{-1}=x^{-1}\rangle .
\]
The generalized quaternion group of order $4n$ is denoted by $Q_{4n}$ and is given by
\[
Q_{4n}=\langle x,y \mid x^{2n}=e,\; y^2=x^n,\; yxy^{-1}=x^{-1}\rangle .
\]
The semi-dihedral group of order $8n$, denoted by $SD_{8n}$, has the presentation
\[
SD_{8n}=\langle x,y \mid x^{4n}=e,\; y^2=e,\; yxy^{-1}=x^{2n-1}\rangle.
\]
A finite group $G$ is called a {$p$-group} if there exists a prime number $p$ such that $|G|=p^n$ for some positive integer $n$.\\

%A group $G$ is called an {AC-group} if for every non-central element $x\in G\setminus Z(G)$, the centralizer $C_G(x)=\{g\in G: gx=xg\}$ is abelian. 

\section{Lambda number of the reduced power graph $\Gamma_{RP}(G)$} \label{sec3}
First, we will study the lambda number of a graph $\Gamma$.

\begin{definition}
    An \emph{$L'(2,1)$-labeling} of a graph $\Gamma$ is an injective function 
$f: V(\Gamma) \to \mathbb{Z}_{\ge 0}$ such that
\[
|f(u)-f(v)| \ge 2 \text{ if } d(u,v)=1, \quad
|f(u)-f(v)| \ge 1 \text{ if } d(u,v)=2.
\]
 The definition of $L(2,1)$ and  $L'(2,1)$-labeling is same, except that the function $f$ is required to be one-to-one.
If $\mathrm{diam}(\Gamma) \le 2$, then $\lambda(\Gamma)= \lambda'(\Gamma)$.
\end{definition}

\noindent Hence, we use $\lambda(\Gamma)$ instead of $\lambda'(\Gamma)$ whenever the graph is of diameter at most two.

\begin{definition}
  Let \(\Gamma\) be a finite graph and \(S \subseteq V(\Gamma)\). Define
\(\operatorname{Gr}[S]\) to be the subgraph of \(\Gamma\) generated by \(S\), which is
obtained from \(\Gamma\) by removing the vertices of \(V(\Gamma)\setminus S\)
and removing the edges in \(E(\Gamma)\) which are adjacent to the vertices of
\(V(\Gamma)\setminus S\).
\end{definition}

 \begin{theorem} \label{lower bound main thm}
      Let $\Gamma$ be a graph of order $n$ with $r$ universal vertices $(1 \leq r < n) $. Then
 \[\lambda(\Gamma) \geq n+r-1.\]
\end{theorem}
 \begin{proof}
      Let \(\Gamma\) be a graph of order \(n\) with \(r\) universal vertices, where \(1\le r<n\). Let \(
H=\operatorname{Gr}[V(\Gamma)\setminus \operatorname{Univ}(\Gamma)].\)
Then,
\[
\Gamma = K_r + H.
\]

By \cite[Lemma 3.4, p.5]{chang19962},
     \begin{align} \label{eq 1}
     \lambda(\Gamma)=& \lambda(K_r + H ) \nonumber \\
         =& \lambda'(K_r) + \lambda'(H) +2  \nonumber \\
         =& \lambda(K_r) + \lambda'(H) +2  \nonumber \\
         =& 2r-2 + \lambda'(H) + 2 \nonumber \\
         =& 2r + \lambda'(H).
     \end{align}

\noindent By the definition of $\lambda'(H)$, we have
\begin{align*}
\lambda'(H) \geq |V(H)|-1
= n-r-1.
\end{align*}

 By Equation (\ref{eq 1}), $\lambda(\Gamma) \geq ~ n+r-1$.
 \end{proof}

 \begin{definition}
The {clique number} of a graph $G$, denoted by $\omega(G)$, is the maximum cardinality of a clique in $G$. Equivalently, $$\omega(G)=\max\{|C| : C \text{ is a clique in } G\},$$ where a clique is a subset of vertices of $G$ such that every two distinct vertices in the subset are adjacent.
\end{definition}

  \begin{lemma}\label{less than formula lemma}
Let \( \Gamma = K_r + \left(K_{m_1}\cup K_{m_2}\cup \cdots \cup K_{m_p}\right),
\qquad p\ge2, \) be a graph of order \(n\) satisfying \(
\omega(\Gamma)\le \frac{n}{2}.\)
For each non-universal vertex \(x_i\in V(\Gamma)\setminus Univ(\Gamma)\), let \(m_i\) denote the number of non-universal vertices adjacent to \(x_i\). Then \( \deg(x_i)=r+m_i, \)
and \( 2m_i-2<n-r-1.\)
\end{lemma}

 \begin{proof}
Since every non-universal vertex belongs to a clique of size \(r+m_i\), we have
\( \deg(x_i)=r+m_i \leq \omega(\Gamma)\le \frac{n}{2}.
\)
Hence,
\( m_i\le \frac{n}{2}-r. \)
Therefore,
\begin{align*}
2m_i-2
    &\le 2\left(\frac{n}{2}-r\right)-2 \\
    &= n-2r-2 \\
    &< n-r-1,
\end{align*}
since \(r\ge 0\). This completes the proof.
\end{proof}

%\begin{example}    For $n \geq 2$, $\lambda(\Gamma(Q_{4n}))= 4n+1$. \end{example}\begin{proof}    To be written later.\end{proof}

    %\subsection{$\lambda(\Gamma)$ of the AC-groups}

     \begin{theorem} \label{thm cluster graphs}
     Let \( \Gamma = K_r + \left(K_{m_1}\cup K_{m_2}\cup \cdots \cup K_{m_p}\right) \), $p \geq 2$ be a graph of order $n$ with clique number \(\omega(\Gamma)\le \frac{n}{2}\). Then the lambda number of the graph $\Gamma$ is $n+r-1$.
 \end{theorem}

 \begin{proof} Note that $$\Gamma = K_r + (K_{m_1} \cup K_{m_2} \cup \hdots \cup K_{m_p}),$$
   where $\displaystyle \sum_{i=1}^p m_i= n -r$ and $m_1 \leq m_2 \leq \hdots  \leq m_p$.
   Then $\Gamma= K_r + H_1$,
   where $H_1=K_{m_1} \cup K_{m_2} \cup \hdots \cup K_{m_p}$.\\

   \noindent By \cite[Lemma 3.4, p.5]{chang19962}, we have
   \begin{align} \label{eq 2}
       \lambda(\Gamma) =& \lambda(K_r + H_1)  \nonumber\\ 
       =& \lambda'(K_r) + \lambda'(H_1) +2  \nonumber\\
       =& \lambda(K_r) + \lambda'(H_1) +2  \nonumber\\
       =& 2r-2 + \lambda'(H_1) + 2  \nonumber \\
       =& 2r + \lambda'(H_1).
   \end{align}

  %\noindent Now, $\lambda'(H_1)= \lambda'(K_{m_1} \cup K_{m_2} \cup \hdots \cup K_{m_p})$.\\

  \noindent By \cite[Lemma 3.2, p.4]{chang19962},
  $$\lambda'(H_1)= max \{ \lambda'(K_{m_1}), \lambda'(H_2), \sum_{i=1}^p m_i-1\},$$

 \noindent  where $H_2= K_{m_2} \cup K_{m_3} \cup \hdots \cup K_{m_p}$.\\

  \noindent Therefore,
 \begin{align} \label{eq 3}
 H_1=& K_{m_1} \cup H_2 \mbox{~and} \nonumber \\
     \lambda'(H_1)=& max \{ 2m_1-2, \lambda'(H_2),  n-r-1\}.
 \end{align} 

 \noindent Then by Lemma \ref{less than formula lemma}, $2m_1-2< n-r-1$. This implies
 \begin{equation} \label{eq 5}
     \lambda'(H_1)= max \{ n-r-1, \lambda'(H_2)\}
 \end{equation}

\noindent Now, \begin{align*}
    \lambda'(H_2)=& \lambda'(K_{m_2} \cup K_{m_3} \cup  \hdots K_{m_p}) \\
    =& \lambda'(K_{m_2} \cup H_3),
\end{align*}

\noindent where $H_3= K_{m_3} \cup K_{m_4} \cup \hdots \cup K_{m_p}$.\\

\noindent By \cite[Lemma 3.2, p.4]{chang19962}, \begin{align} \label{eq 6}
    \lambda'(H_2)=& max  \{ \lambda'(K_{m_2}), \lambda'(H_3), \sum_{i=2}^{p} {m_i}-1\} \nonumber \\
    =& max  \{ 2m_2-2, \lambda'(H_3), n-r-m_1-1\} 
\end{align}

\noindent  By Lemma \ref{less than formula lemma}, $2m_2-2< n-r-1$. Also, $n-r-m_1-1 < n-r-1$. Therefore, by Equation (\ref{eq 5}) and (\ref{eq 6}), we get 
$\lambda'(H_1)= max  \{ n-r-1, \lambda'(H_3)\}.$

\noindent Inductively, we get
\begin{equation}
    \lambda'(H_1)= max  \{ n-r-1, \lambda'(H_{p-1})\},
\end{equation} 
\noindent where $H_{p-1}= K_{m_{p-1}} \cup K_{m_{p}}.$\\

\noindent Now,  $\lambda'(H_{p-1})= \lambda'(K_{m_{p-1}} \cup K_{m_{p}})$. By \cite[Lemma 3.2, p.4]{chang19962}, we have 
\begin{align}
    \lambda'(H_{p-1})=& \lambda'(K_{m_{p-1}} \cup K_{m_{p}}) \nonumber \\
    =& max \{2m_{p-1}-2, 2m_p-2, m_{p-1} + m_{p}-1\} \nonumber \\
    =& max \{2m_p-2, m_{p-1} + m_{p}-1\}
\end{align}

\noindent By Lemma \ref{less than formula lemma}, $2m_p-2 < n-r-1$. Also, $ n-r-1 > m_p + m_{p-1}-1$.\\

\noindent Hence, $\lambda'(H_1)= n-r-1$. Thus,
$$\lambda(\Gamma)= 2r+n-r-1 = n+r-1.$$
\end{proof}

\begin{corr} \label{star}
     Let $\Gamma$ be a star graph of order $n \geq 4$. Then $\lambda(\Gamma)=n$.
\end{corr}

\begin{proof}
    Note that every star graph can be written as $K_1 + ( n-1) K_1$ and clique number $\omega(\Gamma)=2 \leq \frac{n}{2}$. Then by Theorem \ref{thm cluster graphs}, $\lambda(\Gamma)=n$.
\end{proof}

% \begin{theorem}    Let $\Gamma$ be a graph of order $n$ which is a  disjoint union of complete subgraphs with universal vertices $r$  ($1 \leq r <n$). Then the lambda number of the graph $\Gamma$ is $n+r-1$. \end{theorem}

%\begin{proposition}  Let $\Gamma$ be a star graph of order $n$. Then $\lambda(\Gamma)=n$.\end{proposition}\begin{proof}    Define $f: V(\Gamma) \longrightarrow \{0,1,2, \hdots n-1\}$ by  $$f(x_0)=0, f(x_i)=i+1 ; \forall i=1,2, \hdots, n-1,$$ \noindent where $V(\Gamma)= \{x_0,x_1, x_2, \hdots, x_{n-1} \}$, with $x_0$ as the central vertex.  \noindent Since $\Gamma$ is a star graph, $f$ is a $L(2,1)$ labeling with a span equal to $n$. By Theorem \ref{lower bound main thm}, $\lambda(\Gamma)=n$.\end{proof}

Now, we will study the lambda number of the reduced power graph of finite groups. %We define the {reduced power graph} of a group $G$, denoted by $\Gamma_{RP}(G)$, as the simple undirected graph with vertex set $G$, where two distinct vertices $x,y \in G$ are adjacent if and only if $\langle x \rangle \subset \langle y \rangle$ or $\langle y \rangle \subset \langle x \rangle$ and $\langle x \rangle \neq \langle y \rangle$.

\begin{proposition} \cite[Theorem 2.7, p.73]{rajkumar2017reduced}
    The reduced power graph $\Gamma_{RP}(G)$ of a group $G$ is a complete graph if and only if $|G| \leq 2$.
\end{proposition}

%\subsection{Lower bound of the lambda number of the reduced power graph $\Gamma_{RP}(G)$}

In the reduced power graph $\Gamma_{RP}(G)$, $e$ is adjacent to all the remaining vertices, hence by Theorem \ref{lower bound main thm}, we can also find the lower bound of the reduced power graph as given below:

    \begin{proposition} \label{thm lower bound reduced}
      Let $G$ be a group. Then  \(\lambda(\Gamma_{RP}(G)) \geq |G|.\)
\end{proposition}

%\begin{proof}    Since $e$ is adjacent with all other vertices, by Theorem \ref{lower bound main thm}, we can get the wanted result.\end{proof}

\begin{proposition} \label{lower bound reduced power eq to n}
    Let $G$ be a finite non-trivial group. Then $\lambda(\Gamma_{RP}(G))= |G|$ if and only if $\overline{(\Gamma_{RP}(G)-e)}$ contains a Hamiltonian path.
\end{proposition}

\begin{proof}
    Proof is similar to the proof of \cite[Theorem 3.1, p.747]{ma2021lambda}.
\end{proof}

\begin{proposition} \label{pgroup cyclic lambda}
    Let $G$ be a cyclic group of order $p^n$, where $p$ is a prime. Then
    $$\lambda(\Gamma_{RP}(G))=p^n+n-1.$$
\end{proposition}

\begin{proof}
We will prove the statement by the mathematical induction on $n$.  For $n=1$, $G \cong \mathbb{Z}_p$. Then $\Gamma_{RP}(G)$ is a star graph. By Corollary \ref{star}, $\lambda(\Gamma_{RP}(G))=p$.
    \noindent Let for $n=k$, $\lambda(\Gamma_{RP}(\mathbb{Z}_k))=p^k+k-1$. \\
  Now, for $n=k+1$, note that $\Gamma_{RP}(\mathbb{Z}_{p^{k+1}})=  \Gamma_{RP}(\mathbb{Z}_{p^k}) + \overline{K_{\phi(p^{k+1})}},$ where $\phi$ is the Euler phi function. %Since $\overline{K_{\phi(p^{k+1})}}$ is null graph, by \cite[Lemma 3.4, p.5]{chang19962} we use $\lambda'(\overline{K_{\phi(p^{k+1})}})$ instead of $\lambda(\overline{K_{\phi(p^{k+1})}})$ for function to be one-one. \\
  By \cite[Lemma 2.2, p.2]{chang19962}  and \cite[Lemma 3.4, p.5]{chang19962}, we have
    \begin{align*}
       \lambda(\Gamma_{RP}(\mathbb{Z}_{p^{k+1}})) =\lambda'(\Gamma_{RP}(\mathbb{Z}_{p^{k+1}}))=&  \lambda'(\Gamma_{RP}(\mathbb{Z}_{p^k})) + \lambda'(\overline{K_{\phi(p^{k+1})}})+2 \\ =&  \lambda(\Gamma_{RP}(\mathbb{Z}_{p^k})) + \lambda'(\overline{K_{\phi(p^{k+1})}})+2 \\ 
        =& p^k +k-1 +2 + \lambda'(\overline{K_{\phi(p^{k+1})}}) \\
        =& p^k + k+1 + \phi(p^{k+1})-1 \\
        =& p^k + k + p^{k+1} - p^k \\
        =& p^{k+1} + (k+1) -1.
    \end{align*}

  \noindent  Hence, $\lambda(\Gamma_{RP}(G))=p^n+n-1$ for all $n \in \mathbb{N}$.
\end{proof}

\noindent In the following section, we identify some groups for which the $\lambda$-numbers of their reduced power graphs achieve the lower bound stated in Proposition \ref{thm lower bound reduced}.

\begin{example} \label{ex1}
     %Then, \[\lambda(\Gamma_{RP}(D_{2n}))=2n.\]

   Consider the dihedral group  $D_{n}, n \geq 3$  with $|Univ(\Gamma_{RP}(D_{n}))|=1$. 
Write \( B = \{y, xy, x^2y, \ldots, x^{n-1}y\} \). Then 
\(
D_{n} = \langle x \rangle \cup B.
\)

\noindent Define an integer-valued function \( f \) on \( D_{n} \) as follows:
\[
f(e) = 0, \quad f(x) = 3, \quad f(x^2) = 5, \, \ldots, \, f(x^{n-1}) = 2(n-1) + 1,
\]
\[
f(y) = 2, \quad f(xy) = 4, \quad f(x^2y) = 6, \, \ldots, \, f(x^{n-1}y) = 2(n-1) + 2.
\]
Note that the edge set
\[
E(D_{n}) = E(\langle x \rangle) \cup \{\{e, z\} : z \in B\},
\]
which implies that \( f \) is an \(L(2,1)\)-labeling of \( D_{n} \), and so \( \lambda(\Gamma_{RP}(D_{n})) \le 2n \).
By Proposition \ref{thm lower bound reduced}, we obtain \( \lambda(\Gamma_{RP}(D_{n}))=2n. \)

\end{example}

For the generalized quaternion group $Q_{4n}$ with $n \ge 2$, the $|Univ(\Gamma_{RP}(Q_{4n}))|$ of its reduced power graph $\Gamma_{RP}(Q_{4n})$ equals $2$ if $n$ is a power of $2$, and equals $1$ otherwise.

\begin{example} \label{p group qn lambda}
   For $n \ge 2$, we have
\[
\lambda(\Gamma(Q_{4n})) =
\begin{cases}
4n + 1, & \text{if $n$ is a power of $2$},\\[4pt]
4n, & \text{otherwise},
\end{cases}
\]
where $Q_{4n}$ denotes the generalized quaternion group of order $4n$.
\end{example}

\begin{proof}

First, suppose that $n$ is a power of $2$. Write
\(
B=\{\,y,\; xy,\; x^2y,\; \ldots,\; x^{2n-1}y\,\}.
\)
Then
\(
Q_{4n}=\langle x\rangle \cup B.
\)

\noindent Define an integer-valued function \(f\) on \(Q_{4n}\) by
\[
\begin{aligned}
&f(e)=0,\qquad f(x^n)=2,\\
&f(x^k y)=2k+4,\quad\text{for }k=0,1,2,\dots,2n-2, \\
&f(x^k)=2k+3,\quad\text{for }k=1,2,\dots, n-1,\\
&f(x^k)=2k+1,\quad\text{for }k=n+1,n+2,\dots,2n-1  \\
\mbox{~and~}  &f(x^{2n-1}y)= 4n+1.
\end{aligned}
\]

%In particular \(f(y)=3,\; f(x)=4\), and the largest label used is\[f(x^{2n-1}y)=2(2n-1)+3=4n+1.\]

\noindent From the structure of $\Gamma_{RP}(Q_{4n})$ given in \cite{rajkumar2018some}, one can observe that 
%every element of \(B\) has order \(4\) or \(2\) in the quaternion group structure (and \(x^n\) is the central involution). The labelling above assigns even labels to the powers of \(x\) (including the central element \(x^n\)) and odd labels to the coset elements \(x^k y\);
adjacent vertices receive labels differing by at least \(2\), while vertices at distance two receive labels differing by at least \(1\). Hence \(f\) is an \(L(2,1)\)-labeling of \(Q_{4n}\) with span $4n+1$.
Therefore,
\[
\lambda(\Gamma_{RP}(Q_{4n})) \le 4n+1.
\]
Combining this with the lower bound from Theorem \ref{lower bound main thm}, we obtain
\[
\lambda(\Gamma_{RP}(Q_{4n})) = 4n+1.
\]
In the following, suppose that $n$ is not a power of $2$. In view of Theorem \ref{lower bound main thm}, it
suffices to prove that $\lambda(\Gamma_{RP}(Q_{4n}) \leq 4n$.
Note that there exists an element $x_0$ of odd order in $Q_{4n}$.  
Write
\[
\langle x \rangle \setminus \{e, x^n, x_0\} = \{z_1, z_2, \ldots, z_{2n-3}\}.
\]
Define an integer-valued function $f$ on the vertex set $V(Q_{4n})$ as follows:
\[
f(e)=0,\quad f(x^n)=2,\quad f(x_0)=3,
\]
\[
f(z_i)=2(i+1)+1 \quad \text{for } i \in \{1,2,\ldots,2n-3\},
\]
\[
f(x^j y)=2(j+2) \quad \text{and} \quad f(x^{2n-j}y)=2n+2(j+1)
\quad \text{for } j \in \{1,2,\ldots,n-1\},
\]
\[
f(y)=4 \quad \text{and} \quad f(x^n y)=4n-1.
\]
\noindent From the structure of $\Gamma_{RP}(Q_{4n})$ given in \cite{rajkumar2018some}, one can observe that 
adjacent vertices receive labels differing by at least \(2\), while vertices at distance two receive labels differing by at least \(1\). Hence \(f\) is an \(L(2,1)\)-labeling of \(Q_{4n}\) with span $4n$.
Therefore, $\lambda(\Gamma_{RP}(Q_{4n})) \leq 4n$.
\end{proof}

\begin{example} \label{ex3}
    Consider the  semi-dihedral group  \[ SD_{8n} = \langle x, y ~|~ x^{4n} = e, y^2 = x^{2n}, y^{-1}xy = x^{2n - 1} \rangle,\] with $|Univ(\Gamma_{RP}(SD_{8n}))|=1$.
\noindent Define an integer-valued function \(f\) on \(SD_{8n}\) by
\[
\begin{aligned}
&f(e)=0,\qquad f(x^{2n})=2,\\
&f(x^k y)=2k+3,\quad\text{for }k=0,2,4,\dots,4n-4, \\
&f(x^k)=2k+2,\quad\text{for }k=1,2,\dots, 2n-1,\\
&f(x^k y)=4n+k,\quad\text{for }k=1,3,5,\dots,4n-1, \\
&f(x^k)=2k,\quad\text{for }k=2n+1,2n+2,\dots,4n-1  \\
\mbox{~and~}  &f(x^{4n-2}y)= 8n.
\end{aligned}
\]

%In particular \(f(y)=3,\; f(x)=4\), and the largest label used is\[f(x^{2n-1}y)=2(2n-1)+3=4n+1.\]

\noindent From the structure of $\Gamma_{RP}(SD_{8n})$ given in \cite{rajkumar2018some}, one can observe that 
%every element of \(B\) has order \(4\) or \(2\) in the quaternion group structure (and \(x^n\) is the central involution). The labelling above assigns even labels to the powers of \(x\) (including the central element \(x^n\)) and odd labels to the coset elements \(x^k y\);
adjacent vertices receive labels differing by at least \(2\), while vertices at distance two receive labels differing by at least \(1\). Hence \(f\) is an \(L(2,1)\)-labeling of \(SD_{8n}\) with span $8n$.
Therefore,
\[
\lambda(\Gamma_{RP}(SD_{8n})) \le 8n.
\]
Combining this with the lower bound from Proposition \ref{thm lower bound reduced}, we obtain
\[
\lambda(\Gamma_{RP}(SD_{8n})) = 8n.
\]
\end{example}

\begin{theorem}
     \cite{dummit_foote_1991} Let $G$ be a finite non-cyclic nilpotent group. Then $G$ has one of the following forms:
    % the following statements are equivalent:\begin{enumerate} \item $G$ is a nilpotent group. \item Every Sylow subgroup of $G$ is normal. \item $G$ is the direct product of its Sylow subgroups. \item For $x,y \in G$, $x$ and $y$ commute whenever $o(x)$ and $o(y)$ are relatively primes.\end{enumerate}

\begin{enumerate}
    \item $G \cong G' \times \mathbb{Z}_n$, where $G'$ is a non-trivial nilpotent group of odd order having no cyclic Sylow subgroup and $\gcd(n,|G'|)=1$.

    \item $G \cong G' \times P \times \mathbb{Z}_n$, where $G'$ is a nilpotent group of odd order having no cyclic Sylow subgroup, $P$ is a $2$-group which is neither cyclic nor of maximal class and $\gcd(n,|G'|)=\gcd(2,n)=1$.

    \item $G \cong G' \times Q_{2^{k+1}} \times \mathbb{Z}_n$, where $G'$ is described as in (2), $Q_{2^{k+1}}$ is a generalized quaternion group of order $2^{k+1}$ and $\gcd(n,|G'|)=\gcd(2,n)=1$.

    \item $G \cong G' \times D_{2^{k+1}} \times \mathbb{Z}_n$, where $G'$ is described as in (2), $D_{2^{k+1}}$ is a dihedral group of order $2^{k+1}$ and $\gcd(n,|G'|)=\gcd(2,n)=1$.

    \item $G \cong G' \times SD_{2^{k+1}} \times \mathbb{Z}_n$, where $G'$ is described as in (2), $SD_{2^{k+1}}$ is a semidihedral group of order $2^{k+1}$ and $\gcd(n,|G'|)=\gcd(2,n)=1$.
\end{enumerate}
\end{theorem}

\begin{proposition}
 \label{theorem enhanced power reduced}
    For any finite group G, $$\lambda(\Gamma _{RP}(G)) \leq \lambda(\Gamma _{P}(G)) \leq \lambda(\Gamma _{EP}(G)),$$ where $\Gamma _{P}(G)$ is the power graph of $G$ and $\Gamma _{EP}(G)$ is the Enhanced power graph of $G$.
\end{proposition}

\begin{proof}
The result follows by noting the inclusion of the above graphs and \cite[Lemma 2.1, p.2]{chang19962}.
\end{proof}

\begin{proposition} \label{p group rest lambda}
    Let $G$ be a $p$-group. Then, $\lambda(\Gamma_{RP}(G))=|G|$ if and only if $G$ is neither cyclic nor a generalized quaternion $2$-group. 
\end{proposition}
\begin{proof}
    From\cite [Main Theorem]{sarkar2023lambda}, if $G$ is neither cyclic nor generalized quaternion $2$-group, $\lambda(\Gamma_{P}(G))=|G|$.
    Since $\lambda(\Gamma _{RP}(G)) \leq \lambda(\Gamma _{P}(G))$, $\lambda(\Gamma_{RP}(G)) \leq|G|$. By Theorem \ref{lower bound main thm}, we get $\lambda(\Gamma_{RP}(G))=|G|$.
\end{proof}

\begin{theorem} \label{theorem finite p-groups}
For a finite $p$-group $G$,\\
    $$\lambda (\Gamma _{RP}(G))=\begin{cases}
        p^n+n-1 &; G=\mathbb Z_{p^n}\\
        |G|+1 &; \text{if $G$ is Generalized quaternion 2-group} \\
        |G| &; \text{otherwise}
    \end{cases}$$ 
    \end{theorem}

    \begin{proof}
    By Proposition \ref{pgroup cyclic lambda}, Example \ref{p group qn lambda} and Proposition \ref{p group rest lambda}, we get the required result.
\end{proof}

 \begin{proposition}\label{theorem value m=1}
    For any two distinct primes $p$ and $q$, suppose that $m$ is a positive integer such that $m \leq p-1$. Then, $$\lambda(\Gamma _ {RP}(\mathbb Z _{pq^m})=pq^m+1.$$
\end{proposition} 

\begin{proof}
 Let $X = \{x_1, x_2, \hdots, x_{p-1}\}$ denote the set of elements of order $p$. Let $Y^i = \{y_1^i, y_2^i, \hdots, y^i_{\phi(q^i)}\}$ denote the set of elements of order $q^i, i \leq m$. Let $Z^{i} = \{z_1^i, z_2^i, \hdots, z^i_{\phi(pq^i)}\}$ denote the set of elements of order $pq^i, i \leq m-1$.
    %Now, let $X\equiv $ be the set of all elements of order $p$, say $X=\{x_1,x_2,\cdots ,x_{p-1}\}$\\$Y^i\equiv $ the set of all elements of order $q^i$,$\forall i=1,2,\cdots ,m$, say $Y=\{y^i_1,y^i_2,\cdots ,y^i_{\phi(q^i)}\}$.\\$Z^i\equiv $ the set of all elements of order $pq^i$,$\forall i=1,2,\cdots ,m-1$, say $Z=\{z^i_1,z^i_2,\cdots ,z^i_{\phi(pq^i)}\}$.\\
    For $m\leq p-1$, $\overline{\Gamma _{RP}(X \bigcup_{i=1}^{m} Y^i \bigcup_{i=1}^{m-1}  Z^i) }$ contains a Hamiltonian path, that is given by
    \begin{center}
        $x_1 \sim y^1_1\sim y^1_2\sim \cdots \sim y^1_{\phi(q)}\sim x_2\sim y^2_1\sim z^1_1\sim z^1_2\sim \cdots \sim z^1_{\phi(pq)}\sim y^2_2\sim\cdots \sim y^2_{\phi(q^2)}\sim x_3\sim\cdots  \sim y^{m-1}_{\phi(pq^{m-1}}\sim x_m\sim y^m_1\sim z^{m-1}_1\sim z^{m-1}_2\sim\cdots \sim z^{m-1}_{\phi(pq^{m-1})}\sim y^m_2\sim\cdots \sim y^m_{\phi(q^m)}\sim x_{m+1}\sim \cdots \sim x_{\phi (p)}\sim x_1$.
    \end{center}
    The connected components of $\overline{\Gamma _ {RP}(\mathbb Z _{pq^m})}$ are $K_1$, $\overline{\Gamma _{RP}(X \bigcup_{i=1}^{m} Y^i \bigcup_{i=1}^{m-1}  Z^i) }$ and $K_{\phi (pq^m)}$, that is  $c(\overline{\Gamma _ {RP}(\mathbb Z _{pq^m})})=3$. By Lemma \ref{thm 1}, $\lambda(\Gamma _ {RP}(\mathbb Z _{pq^m})=pq^m+1$ for $m \leq p-1$.
\end{proof}

\begin{corr} \label{corr pq}
    Let $G$ be a group of order $pq$, where $p$ and $q$ are distinct primes. Then, 
   $ \lambda (\Gamma _{RP}(G))=\begin{cases}
        pq+1, &\text{when $G$ is cyclic.} \\
        pq, & \text{when $G$ is non-abelian} 
    \end{cases}$
\end{corr}
    
\begin{proof}
    If $G$ is a non-abelian group of order $pq$, then $\Gamma _{RP}(G)$ is a star graph. By Corollary \ref{star}, $\lambda (\Gamma _{RP}(G))=pq$.\\
    If $G$ is cyclic, taking $m=1$ in Proposition \ref{theorem value m=1}, we get $\lambda (\Gamma _{RP}(\mathbb Z_{pq})=pq+1$.
\end{proof}

\begin{proposition}
   Let $G$ be a non-trivial finite simple group of order $n$. Then $\lambda(\Gamma_{RP}(G))=n$.
\end{proposition}

\begin{proof}
    Let $G$ be a non-trivial finite simple group of order $n$. From \cite[Theorem 3.6]{parveen2025lambda}, if $G$ is not a cyclic of order $n\geq 3$, then $\lambda(\Gamma_{EP}(G))=n$. By Proposition \ref{theorem enhanced power reduced}, we get $\lambda(\Gamma_{RP}(G)) \leq n$. Now, by Proposition \ref{thm lower bound reduced}, we get $\lambda(\Gamma_{RP}(G))=n$. If $G$ is a cyclic-simple group, then by Theorem \ref{theorem finite p-groups}, $\lambda(\Gamma_{RP}(G))=n$.
\end{proof}

\begin{proposition}
    Let $G'$ be a non-trivial nilpotent group of odd order having no cyclic Sylow subgroup.If $G \cong G' \times \mathbb{Z}_n$,where $n \geq 1$ and $\gcd(n,|G'|)=1$. Then $\lambda(\Gamma_{RP}(G))=|G|+r-1$, where $r$ is the number of universal vertices of $\Gamma_{RP}(G)$.
\end{proposition}

\begin{proof}
     From \cite[Theorem 3.13]{parveen2025lambda}, $\lambda(\Gamma_{EP}(G))=|G|+r-1$. By Proposition \ref{theorem enhanced power reduced}, we get $\lambda(\Gamma_{RP}(G)) \leq |G|+r-1$. Now, by Theorem \ref{lower bound main thm}, we get $\lambda(\Gamma_{RP}(G))=|G|+r-1$.
\end{proof}

\begin{proposition}
     Let $G'$ be a nilpotent group of odd order having no cyclic Sylow subgroup. If $G \cong G' \times P \times \mathbb{Z}_n$, where $P$ is a non-cyclic $2$-group  and $\gcd(n,|G'|)=\gcd(2,n)=1$, then $\lambda(\Gamma_{RP}(G))=|G|+r-1$, where $r$ is the number of universal vertices of $\Gamma_{RP}(G)$.
\end{proposition}

\begin{proof}
     From \cite[Theorem 3.16]{parveen2025lambda}, $\lambda(\Gamma_{EP}(G))=|G|+r-1$. By Proposition\ref{theorem enhanced power reduced}, we get $\lambda(\Gamma_{RP}(G)) \leq |G|+r-1$. Now, by Theorem \ref{lower bound main thm}, we get $\lambda(\Gamma_{RP}(G))=|G|+r-1$.
\end{proof}

\begin{theorem}\label{upper bound main thm}
  Let \(G\) be a finite group such that \(G\not\cong \mathbb{Z}_{p^r}\) for any prime \(p\) and any positive integer \(r\). Then $\lambda(\Gamma_{RP}(G)) \leq 2n-4$ with equality if and only if $G$ is isomorphic to $\mathbb Z_2 \times \mathbb Z_2$. 
\end{theorem}

\begin{proof}
    Let $G$ be a finite group other than $ \mathbb{Z}_{p^r}, \mathbb Z_2 \times \mathbb Z_2, \mathbb Z_{2q}$, where $p$ is a prime, $r$ a
positive integer and $q$ an odd prime. Then from \cite[Theorem 4.1]{ma2021lambda}, $\lambda(\Gamma_{P}(G)) < 2n-4$. By Proposition \ref{theorem enhanced power reduced}, we get $\lambda(\Gamma_{RP}(G)) < 2n-4.$
    From Corollary \ref{corr pq}, $\lambda (\Gamma _{RP}(\mathbb Z_{2q})=2q+1<2n-4$. Since $\Gamma _{RP} (\mathbb Z_2 \times \mathbb Z_2)$ is a star graph, $\lambda(\Gamma _{RP} (\mathbb Z_2 \times \mathbb Z_2))=4$.
\end{proof}

%\begin{proposition} \label{upper bound second thm} Let $G$ be a cyclic group of order $n=p^r$. Then, $\lambda(\Gamma _{RP} (G))=2n-3$ if and only if $G=\mathbb Z_3$ or $G=\mathbb Z_4$. \end{proposition} \begin{proof}  From Theorem \ref{upper bound main thm}, it is clear that $\lambda(\Gamma _{RP} (G))=2n-3$ is possible for a cyclic group of prime power order. Now, suppose that    $$ \lambda(\Gamma _{RP}(\mathbb Z_{p^r}))= p^r+r-1=2n-3 =2p^r-3.$$ Therefore, $ p^r = r+2.$  is possible only for ($p=2,r=2$) and ($p=3,r=1$), that is  $G=\mathbb Z_3$ and $G=\mathbb Z_4$ \end{proof}

    \section*{Declarations}
	
	\noindent \textbf{Acknowledgment}: The first author is supported by senior research fellowship of CSIR, India.\\
	
	\noindent\textbf{Conflicts of interest}: There is no conflict of interest regarding the publishing of this paper.
    
%\noindent\textbf{Availability of data and material (data transparency)}: Not applicable.\\

%\noindent \textbf{Code availability}: Not applicable.

 \small \bibliographystyle{unsrt}
  \bibliography{name}
\end{document}